\documentclass{amsart}

\usepackage{amsfonts}

\usepackage{graphicx}

\usepackage{url}

\DeclareGraphicsExtensions{.pdf,.png,.jpg}
\usepackage{xcolor}
\usepackage{tikz}
\usetikzlibrary{arrows}

\theoremstyle{plain}
\newtheorem{Thm}{Theorem}

\newtheorem{Prop}[Thm]{Proposition}
\newtheorem{Lem}{Lemma}

\theoremstyle{definition}
\newcommand{\sgn}{\operatorname{sgn}}
\newcommand{\spa}{\operatorname{span}}

\allowdisplaybreaks[3]

\begin{document}

\date{\today}

\title[Fourier--Jacobi expansions in Morrey spaces]{%
Fourier--Jacobi expansions in Morrey spaces}

\author[A. Arenas]{Alberto Arenas}
\address{Departamento de Matem\'aticas y Computaci\'on,
Universidad de La Rioja, 
Calle Luis de Ulloa s/n, 26004 Logro\~no, Spain}
\email{alarenas@unirioja.es}

\author[\'O. Ciaurri]{\'Oscar Ciaurri}
\address{Departamento de Matem\'aticas y Computaci\'on,
Universidad de La Rioja, 
Calle Luis de Ulloa s/n, 26004 Logro\~no, Spain}
\email{oscar.ciaurri@unirioja.es}
\thanks{Research of the author supported by grant MTM2012-36732-C03-02 of the DGI}


\keywords{Fourier--Jacobi expansions, Morrey spaces, weighted
inequalities.}
\subjclass[2010]{Primary 42C10, 43A50}

\begin{abstract}
In this paper we obtain a characterization of the convergence of
the partial sum operator related to Fourier--Jacobi expansions in
Morrey spaces.
\end{abstract}

\maketitle

\section{Introduction and main results}
For $\alpha,\beta>-1$, we consider the Jacobi functions
\[
p_n^{(\alpha,\beta)}(x)=d_{n}^{(\alpha,\beta)}P_n^{(\alpha,\beta)}(x)(1-x)^{\alpha/2}(1+x)^{\beta/2},
\quad x\in (-1,1), \quad n=0,1,2,\dots,
\]
where $P_n^{(\alpha,\beta)}$ denotes the Jacobi polynomial of
order $(\alpha,\beta)$ and degree $n$, and
\[
\frac{1}{d_n^{(\alpha,\beta)}}=\left(\int_{-1}^{1}\Big(P_n^{(\alpha,\beta)}(x)\Big)^2(1-x)^{\alpha}(1+x)^{\beta}\,
dx\right)^{1/2}.
\]
The system of functions $\{p_n^{(\alpha,\beta)}\}_{n\ge 0}$ is
orthonormal and complete in $L^2(-1,1)$ with the Lebesgue measure. Given an appropriate function $f$, its Fourier expansion respect to the Jacobi functions, which we call Fourier--Jacobi expansion,
is given by
\[
f\sim \sum_{k=0}^\infty a_k^{(\alpha,\beta)}(f)
p_k^{(\alpha,\beta)},
\qquad
a_k^{(\alpha,\beta)}(f)=\int_{-1}^1 p_k^{(\alpha,\beta)}(t)f(t)\,
dt.
\]

The convergence of the partial sum operator for the
Fourier--Jacobi expansions, given by
\[
S_nf=\sum_{k=0}^n a_k^{(\alpha,\beta)}(f)p_k^{(\alpha,\beta)},
\]
which is equivalent to the uniform boundedness of $S_n$, 
has been widely analyzed in different kinds of spaces. In the case of the Fourier series related to the Jacobi polynomials in $L^p$ spaces the first known results are due to Pollard \cite{Pollard-III} and Wing \cite{Wing} who
treated the case $\alpha,\beta \ge -1/2$. In \cite{Muck}, Muckenhoupt extended the analysis to the whole range $\alpha,\beta>-1$ and included some weights. With the sufficient conditions on the weights given in Muckenhoupt's paper  and by using the results about the necessary conditions in, for instance, \cite{MNT} (about this question see \cite{Cha,Mea,GPRV} also), the result for Fourier--Jacobi functions establishes that
\[
\|S_nf\|_{L^p(-1,1)}\le C \|f\|_{L^p(-1,1)}\iff \frac{4}{3}<p<4.
\]
A complete study of the boundedness properties of the partial sum operator in $L^{p,\infty}$ spaces can be seen in \cite{GPV}.

Our target in this paper is the study of the convergence of the
Fourier--Jacobi expansions in Morrey spaces. To this end, for
$1\le p<\infty$ and $0\le \lambda<1$ we define the Morrey space
$\mathcal{L}^{p,\lambda}(-1,1)$ as the set of functions $f$ on
$(-1,1)$ such that
\[
\|f\|_{p,\lambda}:=\sup_{x\in
(-1,1),r>0}\left(\frac{1}{r^\lambda}\int_{B(x,r)}|f(t)|^p\,
dt\right)^{1/p}<\infty,
\]
where $B(x,r)=\{t\in (-1,1):|t-x|\le r\}$. It is clear that $\mathcal{L}^{p,\lambda}(-1,1)$ are Banach spaces. Morrey spaces can be defined in a more general way but this
is enough for our purposes. The $L^p(-1,1)$ spaces with the Lebesgue measure correspond with the case $\lambda=0$.

Morrey spaces were introduced by Morrey, see \cite{Morrey}, in the setting of partial differential equations. In the last years, some classical operators from harmonic analysis have been analyzed in the setting of Morrey spaces, see, for instance, \cite{Adams}, \cite{Ros}, \cite{Samko-1}, \cite{Samko-2}, and the references therein.

Specially relevant for our purposes will be the boundedness with weights of the Hilbert transform given in \cite[Theorem 4.7 and Corollary 4.8]{Samko-1}. In particular, we will use the following version of that result: if
\[
Hf(x)=\int_{-1}^1 \frac{f(t)}{x-t}\, dt
\]
and
\[
w_k(x)=\prod_{j=1}^k |x-x_j|^{\gamma_j},
\]
with $-1\le x_1 <x_2<\cdots< x_{k-1}<x_k\le 1$, then for $1<p<\infty$ and $0\le \lambda<1$ it holds that
\begin{equation}
\label{eq:Hilbert}
\|w_kHf\|_{p,\lambda}\le C \|w_k f\|_{p,\lambda} \iff
\frac{\lambda-1}{p}<\gamma_j<1+\frac{\lambda-1}{p},\qquad j=1,\dots,k.
\end{equation}

Our result about convergence of Fourier-Jacobi expansions on Morrey spaces is the following.
\begin{Thm}
\label{thm:main}
Let $0\le \lambda <3/4$, $1< p<\infty$, and $\alpha,\beta \ge 0$. Then
\[
S_n f\longrightarrow f, \quad\text{in $\mathcal{L}^{p,\lambda}(-1,1)$},
\]
if and only if
\[
\frac{4}{3}(1-\lambda)<p<4(1-\lambda).
\]
\end{Thm}

As a first step to prove Theorem \ref{thm:main} we start by establishing the equivalence between the convergence of the partial sums and the uniform boundedness of the operator $S_n$ in Morrey spaces.
\begin{Thm}
\label{thm:conve}
Let $0\le \lambda <1$, $1\le p<\infty$ and $\alpha,\beta > -1$. Then
\[
S_n f\longrightarrow f, \quad\text{in $\mathcal{L}^{p,\lambda}(-1,1)$},
\]
if and only if
\[
\|S_nf\|_{p,\lambda}\le C \|f\|_{p,\lambda},
\]
where $C$ is a constant independent of $n$ and $f$.
\end{Thm}
So, Theorem \ref{thm:main} will follow from the uniform boundedness of the partial sum operator in Morrey spaces. The next theorem contains a characterization of the interval of values of $p$ for which this estimate holds.
\begin{Thm}
\label{thm:bound}
Let $0\le \lambda <3/4$, $1< p<\infty$, and $\alpha,\beta \ge 0$. Then
\begin{equation}
\label{eq:Sn}
\|S_nf\|_{p,\lambda}\le C \|f\|_{p,\lambda},
\end{equation}
where $C$ is a constant independent of $n$ and $f$, if and only if
\begin{equation}
\label{eq:condi}
\frac{4}{3}(1-\lambda)<p<4(1-\lambda).
\end{equation}
\end{Thm}

The region where the partial sum $S_n$ converges in Morrey spaces is the shadowed one shown in Figure \ref{fig:Morrey-graf}. Outside and in the border of that region the convergence is not possible. As it occurs in $L^p$ spaces, we think that in the border, with dashed lines in the figure and corresponding with the values $p=4(1-\lambda)$ and $p=\max\{4(1-\lambda)/3,1\}$, some kind of weak boundedness of the partial sum operator should hold.

\begin{center}
\label{fig:Morrey-graf}
\begin{tikzpicture}[scale=2]
\fill[black!10!white] (2,3.1) -- (4.25,0.85) -- (2.75,0.85) -- (2,1.1) -- cycle;
\draw[thick, dashed] (2,3.1) -- (4.25,0.85)-- (2.75,0.85) -- (2,1.1);
\draw[thick, dotted] (4.25,0.85) -- (5,0.1) -- (2.75,0.85) -- (2,0.85);
\draw[very thin] (1.97,0.1) -- (5.1,0.1);
\draw[very thin] (2,0.07) -- (2,3.2);
\draw[very thin] (2.75,0.07) -- (2.75,0.13);
\draw[very thin] (4.25,0.07) -- (4.25,0.13);
\draw[very thin] (5,0.07) -- (5,0.13);
\draw[very thin] (1.97,0.85) -- (2.03,0.85);
\draw[very thin] (1.97,1.1) -- (2.03,1.1);
\draw[very thin] (1.97,3.1) -- (2.03,3.1);
\node at (2,-0.02) {$0$};
\node at (2.75,-0.065) {$\tfrac{1}{4}$};
\node at (4.25,-0.065) {$\tfrac{3}{4}$};
\node at (5,-0.02) {$1$};
\node at (5.17,0.1) {$\lambda$};
\node at (1.9,0.85) {$1$};
\node at (1.9,1.1) {$\tfrac{4}{3}$};
\node at (1.9,3.1) {$4$};
\node at (2,3.3) {$p$};
\draw (3.5,1.75) node [rotate=-45]  {\tiny{$p=4(1-\lambda)$}};
\draw (3.25,0.57) node [rotate=-18.43]  {\tiny{$p=\frac{4}{3}(1-\lambda)$}};
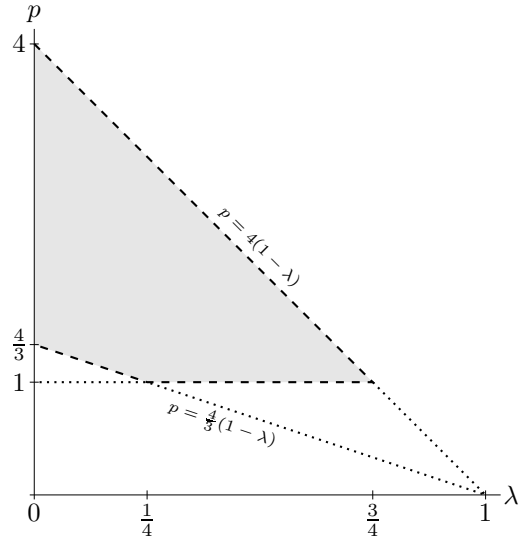
\draw (3.7,-0.4) node {Figure 1: The region where the partial sum operator converges.};
\end{tikzpicture}
\end{center}

The proofs of Theorem \ref{thm:conve} and Theorem \ref{thm:bound} will be contained in Sections \ref{sec:conve} and \ref{sec:bound}, respectively. In the last section, we include the proof of some auxiliary results.
\section{Proof of Theorem \ref{thm:conve}}
\label{sec:conve}
Theorem \ref{thm:conve} is a standard consequence of Banach-Steinhaus Theorem and the following proposition.
\begin{Prop}
\label{prop:conve}
Let $1\le p<\infty$, $0\le \lambda<1$, and $\alpha,\beta>-1$. Then, $\spa\{p_n^{(\alpha,\beta)}\}$ is dense in $\mathcal{L}^{p,\lambda}(-1,1)$.
\end{Prop}
\begin{proof}
The case $\lambda=0$ is well known and our proof for $0<\lambda<1$ relies on it.

Let us suppose that $\spa\{p_n^{(\alpha,\beta)}\}$ is not dense in $\mathcal{L}^{p,\lambda}(-1,1)$, for $0<\lambda<1$. Then, by a standard consequence of Hahn-Banach Theorem \cite[Theorem 5.19]{Rudin}, there exists a non-zero functional $T$ on $\mathcal{L}^{p,\lambda}(-1,1)$ such that $Tp_n^{(\alpha,\beta)}=0$, for all $n\ge 0$.

It is easy to check that for $p\le r$ and $\lambda\le \mu$, we have
\[
\mathcal{L}^{r,\mu}(-1,1)\subset \mathcal{L}^{p,\lambda}(-1,1) \subset \mathcal{L}^{p,0}(-1,1)=L^p(-1,1).
\]
So, by Hahn-Banach Theorem \cite[Theorem 5.16]{Rudin}, we can extend the functional $T$ to $L^p(-1,1)$. Then, there exists a unique function $g\in L^q(-1,1)$, with $p^{-1}+q^{-1}=1$, such that
\[
Tf=\int_{-1}^1 f(x)g(x)\, dx,\qquad \text{for each $f\in L^p(-1,1)$}.
\]
But the condition $0=Tp_n^{(\alpha,\beta)}$ implies $a_n^{(\alpha,\beta)}(g)=0$ and $g=0$. Then $T=0$ and this is a contradiction because $T$ was non-zero.
\end{proof}

\begin{proof}[Proof of Theorem \ref{thm:conve}]
By the Banach-Steinhaus Theorem \cite[Theorem 5.8]{Rudin}, the convergence implies the uniform boundedness.

On the other hand, for each $f\in \mathcal{L}^{p,\lambda}(-1,1)$ and given $\varepsilon >0$, by Proposition \ref{prop:conve}, there exists $g\in \spa\{p_n^{(\alpha,\beta)}\}$ such that $\|f-g\|_{p,\lambda}<\varepsilon$. Moreover there exists $N>0$ such that $S_n(g)=g$, for each $n\ge N$. Then
\begin{align*}
\|S_nf-f\|_{p,\lambda}&\le \|S_nf-S_ng\|_{p,\lambda}+\|g-f\|_{p,\lambda}\\&\le (C+1)\|g-f\|_{p,\lambda}\le (C+1)\varepsilon,
\end{align*}
for $n\ge N$ and $S_n f\longrightarrow f$ in $\mathcal{L}^{p,\lambda}(-1,1)$.
\end{proof}
\section{Proof of Theorem \ref{thm:bound}}
\label{sec:bound}
Before starting with the proof we are going to collect some facts that we will use.

Given a function $g$ defined in $(-1,1)$, for each $1\le p <\infty$ and $0<\lambda<1$, we define
\[
\|g\|_{q,\lambda}^{*}:=\inf_{x\in (-1,1)}\int_0^\infty r^{\lambda/p-1} \|\chi_{(B(x,r))^c}g\|_{L^q(-1,1)}\, dr,
\]
where $(B(x,r))^c=(-1,1)\setminus B(x,r)$ and $p^{-1}+q^{-1}=1$.

With the previous definition we can give an appropriate version of H\"older inequality for Morrey spaces (see \cite[Lemma 4.1]{Go-Mus}). In our case it reads so.
\begin{Lem}
Let $1\le p<\infty$, $p^{-1}+q^{-1}=1$, and $0<\lambda<1$. Then the inequality
\begin{equation}
\label{eq:des-L1}
\int_{-1}^{1}|f(x)g(x)|\, dx\le C \|f\|_{p,\lambda}\|g\|_{q,\lambda}^{*},
\end{equation}
holds with a constant $C$ independent of $f$ and $g$.
\end{Lem}

It is interesting to know some facts about the norm of certain functions in Morrey spaces. From \cite[Remark 4.4]{Samko-2} we see that, for $a\in (-1,1)$ and $\nu p>-1$, the function $|x-a|^{\nu}\chi_{B(x_0,r_0)}\in \mathcal{L}^{p,\lambda}(-1,1)$ if and only if $\nu p \ge \lambda -1$, with $\lambda>0$, and
\begin{equation}
\label{eq:norm-Bola}
\||x-a|^{\nu}\chi_{B(x_0,r_0)}\|_{p,\lambda}\simeq |B(x_0,r_0)|^{(1+\nu p-\lambda)/p}.
\end{equation}

With respect to the functions $p_n^{(\alpha,\beta)}$ we prove the following result.
\begin{Lem}
\label{lem:norms-p}
Let $1\le p<\infty$ and $\alpha,\beta\ge 0$. Then
\begin{equation}
\label{eq:norma-Morrey}
\|p_n^{(\alpha,\beta)}\|_{p,\lambda}\simeq
\begin{cases}
1, &\text{ if $p\le 4(1-\lambda)$},\\
n^{1/2-2(1-\lambda)/p}, &\text{ if $p> 4(1-\lambda)$},
\end{cases}
\end{equation}
for $0< \lambda\le 3/4$, and
\begin{equation}
\label{eq:norma-Lp}
\|p_n^{(\alpha,\beta)}\|_{p,0}\simeq
\begin{cases}
1, &\text{ if $p<4$},\\
(\log n)^{1/4}, & \text{ if $p=4$},\\
n^{1/2-2/p}, &\text{ if $p> 4$}.
\end{cases}
\end{equation}
\end{Lem}

The estimates in the previous lemma will be deduced by using a very sharp bound for the Jacobi polynomials arising from a Hilb type formula for Jacobi polynomials (see \cite[Theorem 8.21.12]{Szego}). In fact, we have
\begin{equation}
\label{eq:Jaco-Hilb}
n^{1/2}|P_n^{(\alpha,\beta)}(x)|\le C (1-x+n^{-2})^{-\alpha/2-1/4}(1+x+n^{-2})^{-\beta/2-1/4},
\end{equation}
for $\alpha,\beta>-1$. Then, for $\alpha,\beta\ge 0$, we deduce in an easy way (note that $d_n^{(\alpha,\beta)}\sim n^{1/2}$) the bounds
\begin{equation}
\label{eq:Jaco-cota-h}
|p_n^{(\alpha,\beta)}(x)|\le C(h_{+,n}(x)+h_{-,n}(x)),
\end{equation}
where $h_{\pm,n}(x)=(1\pm x+n^{-2})^{-1/4}$, and
\begin{equation}
\label{eq:Jaco-cota}
|p_n^{(\alpha,\beta)}(x)|\le C(1-x^2)^{-1/4},
\end{equation}
for $\alpha,\beta\ge0$.

To complete the boundedness of $S_n$, we will use Lemma \ref{lem:norms-p} and to that end we have to estimate $\|g\|_{q,\lambda}^{*}$ for the functions $h_{\pm,n}$. That is the content of the next lemma.
\begin{Lem}
\label{lem:norms-q}
Let $1\le p<\infty$ and $\alpha,\beta>0$.
Then $\|h_{\pm,n}\|_{q,\lambda}^{*}\le C$,
for $p\ge 4(1-\lambda)/3$.
\end{Lem}

Moreover, to analyze the norm of some functions in Morrey spaces, we have to recall the following theorem from \cite{MNT}.
\begin{Thm}
Let $d\mu=w(x)\, dx$ be a measure on $(-1,1)$. If $\{p_n\}_{n\ge 0}$ is the sequence of ortonormal polynomials on $L^2((-1,1),d\mu)$, then
\[
\Big(\int_{-1}^1\Big|\frac{g(x)}{w(x)^{1/2}(1-x^2)^{1/4}}\Big|^{p}w(x)\,dx\Big)^{1/p}\le C \liminf_{n\to \infty} \Big(\int_{-1}^1|p_n(x) g(x)|^p w(x)\,dx\Big)^{1/p},
\]
for $0<p\le \infty$ and for each measurable function $g$ on $(-1,1)$.
\end{Thm}
As an immediate consequence of the previous result, we have
\begin{equation}
\label{eq:MNT}
\|(1-x^2)^{-1/4} g(x)\|_{p,\lambda}\le C \liminf_{n\to \infty} \|p_n^{(\alpha,\beta)}(x) g(x)\|_{p,\lambda}.
\end{equation}

The last tool that we need to complete the proof of the theorem is an easy observation related to the boundedness of the Hilbert transform with weights that will be used to complete the necessity of the conditions \eqref{eq:condi}.

\begin{Lem}
\label{lem:Hil-Mario}
Let $H$ be the Hilbert transform on $(-1,1)$ and let us suppose that
\[
\|uHg\|_{p,\lambda}\le C \|v g\|_{p,\lambda}, \qquad vg\in \mathcal{L}^{p,\lambda}(-1,1),
\]
for two nonnegative weights $u$ and $v$. Then for $r,s \in (-1,1)$ such that $r\le s$ we have the inequality
\begin{equation}
\label{eq:Hil-2}
\left(\int_{-1}^{r}\frac{|f(t)|}{s-t}\, dt\right)\|u\chi_{(r,s)}\|_{p,\lambda}\le C \|v f\chi_{(-1,r)}\|_{p,\lambda},\qquad vf\in \mathcal{L}^{p,\lambda}(-1,1),
\end{equation}
with the same constant $C$.
\end{Lem}

The proofs of Lemma \ref{lem:norms-p}, Lemma \ref{lem:norms-q}, and Lemma \ref{lem:Hil-Mario} will be done in the last section.

\begin{proof}[Proof of Theorem \ref{thm:bound}]
Let us start by assuming the conditions \eqref{eq:condi} to prove the uniform boundedness of the partial sum operators $S_n$. They can be written as
\[
S_nf(x)=\int_{-1}^1 K_n(x,t)f(t)\, dt,
\]
where
\[
K_n(x,t)=\sum_{k=0}^n p_{k}^{(\alpha,\beta)}(x)p_{k}^{(\alpha,\beta)}(t).
\]
Now, we consider Pollard decomposition of the kernel, see \cite{Pollard-II},
\[
K_n(x,t)=a_n T_1(n,x,t)+a_n T_2(n,x,t)+b_n T_3(n,x,t),
\]
where
\[
T_1(n,x,t)=p_{n+1}^{(\alpha,\beta)}(x)\frac{(1-t^2)^{1/2}p_{n}^{(\alpha+1,\beta+1)}(t)}{x-t},
\]
$T_2(n,x,t)=T_1(n,t,x)$ and $T_3(n,x,t)=p_{n+1}^{(\alpha,\beta)}(x)p_{n+1}^{(\alpha,\beta)}(t)$. Moreover the sequences $a_n$ and $b_n$ are bounded. This leads to
\[
S_nf(x)=a_n W_{1,n}f(x)-a_nW_{2,n}f(x)+b_n W_{3,n}f(x),
\]
with
\[
W_{1,n}f(x)=p_{n+1}^{(\alpha,\beta)}(x)H((1-(\cdot)^2)^{1/2}p_{n}^{(\alpha+1,\beta+1)}f)(x),
\]
\[
W_{2,n}f(x)=(1-x^2)^{1/2}p_{n+1}^{(\alpha+1,\beta+1)}(x)H(p_{n}^{(\alpha,\beta)}f)(x),
\]
and
\[
W_{3,n}f(x)=p_{n+1}^{(\alpha,\beta)}(x)\int_{-1}^{1}p_{n+1}^{(\alpha,\beta)}(t)f(t)\, dt.
\]

Now, taking into account the estimate \eqref{eq:Jaco-cota}, the boundedness in $\mathcal{L}^{p,\lambda}(-1,1)$ of $W_{1,n}$ and $W_{2,n}$ will follow from the inequalities
\[
\|(1-x^2)^{\pm 1/4}Hg(x))\|_{p,\lambda}\le C \|(1-x^2)^{\pm 1/4}g(x))\|_{p,\lambda}.
\]
Then, by using \eqref{eq:Hilbert}, it is enough that
\[
\frac{\lambda-1}{p}<\frac{-1}{4}<1+\frac{\lambda-1}{p}\qquad \text{ and }\qquad
\frac{\lambda-1}{p}<\frac{1}{4}<1+\frac{\lambda-1}{p},
\]
and this is implied by \eqref{eq:condi}.

To treat $W_{3,n}$ we start by using \eqref{eq:des-L1}, then
\begin{equation*}
\|W_{3,n}f\|_{p,\lambda}\le \|p_{n+1}^{(\alpha,\beta)}\|_{p,\lambda}\int_{-1}^{1}|p_{n+1}^{(\alpha,\beta)}(t)f(t)|\, dt\le
\|p_{n+1}^{(\alpha,\beta)}\|_{p,\lambda}\|p_{n+1}^{(\alpha,\beta)}\|_{q,\lambda}^{*}\|f\|_{p,\lambda}.
\end{equation*}
Now, from \eqref{eq:Jaco-Hilb}, we have $\|p_{n+1}^{(\alpha,\beta)}\|_{q,\lambda}^{*}\le C (\|h_{-,n+1}\|_{q,\lambda}^{*}
+\|h_{+,n+1}\|_{q,\lambda}^{*})$. So, to conclude the estimate it is enough to apply Lemma \ref{lem:norms-p} and Lemma \ref{lem:norms-q}. Indeed,
\begin{align*}
\|W_{3,n}f\|_{p,\lambda}&\le C
\|p_{n+1}^{(\alpha,\beta)}\|_{p,\lambda}(\|h_{-,n+1}\|_{q,\lambda}^{*}
+\|h_{+,n+1}\|_{q,\lambda}^{*})\|f\|_{p,\lambda}\le C \|f\|_{p,\lambda}.
\end{align*}

Let us show that conditions \eqref{eq:condi} are necessary for the uniform boundedness. We consider the operators $T_nf(x)=S_nf(x)-S_{n-1}f(x)$, then it is clear that
\[
T_nf(x)=a_n^{(\alpha,\beta)}(f)p_n^{(\alpha,\beta)}(x).
\]
By using the uniform boundedness for $S_n$, we have
\begin{equation}
\label{eq:T-op}
\|T_nf\|_{p,\lambda}=|a_n^{(\alpha,\beta)}(f)|\|p_n^{(\alpha,\beta)}\|_{p,\lambda}\le C \|f\|_{p,\lambda},
\end{equation}
with a constant $C$ independent of $n$ and $f$.

Taking $f=\sgn(p_n^{(\alpha,\beta)})$, \eqref{eq:T-op} becomes
\[
\|p_n^{(\alpha,\beta)}\|_{1,0}\|p_n^{(\alpha,\beta)}\|_{p,\lambda}\le C.
\]
So, by using \eqref{eq:norma-Lp}, the inequality $\|p_n^{(\alpha,\beta)}\|_{p,\lambda}\le C$ has to be verified and this fact implies, by \eqref{eq:norma-Morrey}, $p\le 4(1-\lambda)$.

Now, by considering $f=\sgn(p_n^{(\alpha,\beta)})|p_n^{(\alpha,\beta)}|^{q-1}$, with $p^{-1}+q^{-1}=1$, in \eqref{eq:T-op}, we have
\[
\|p_n^{(\alpha,\beta)}\|_{q,0}^{q}\|p_n^{(\alpha,\beta)}\|_{p,\lambda}\|p_n^{(\alpha,\beta)}\|_{q,\lambda}^{-q/p}\le C.
\]
Analyzing the different cases of the previous inequality for $p\le 4(1-\lambda)$ with the estimates in Lemma \ref{lem:norms-p}, we deduce the restriction $p\ge 4(1-\lambda)/3$.

To conclude the necessity of \eqref{eq:condi}, we have to check that for the cases $p=4(1-\lambda)/3$ and $p=4(1-\lambda)$ the inequality \eqref{eq:Sn} is not possible.

For $p=4(1-\lambda)$ the operators $W_{2,n}$ and $W_{3,n}$ are uniformly bounded. Let us see that $W_{1,n}$ is unbounded. We proceed by contradiction. If $W_{1,n}$ is bounded, we have the equivalent estimate
\[
\|p_{n+1}^{(\alpha,\beta)}Hg\|_{p,\lambda}\le C \|(1-(\cdot)^2)^{-1/2}(p_{n}^{(\alpha+1,\beta+1)})^{-1}g\|_{p,\lambda}.
\]
Now, from \eqref{eq:Hil-2} with $s=1$ and $f(t)=\chi_{(0,r)}(t)p_n^{(\alpha+1,\beta+1)}(t)(1-t^2)^{1/4}$, for  $r>0$, we have
\[
\int_{0}^{r} \frac{p_n^{(\alpha+1,\beta+1)}(t)(1-t^2)^{1/4}}{1-t}\, dt \|p_n^{(\alpha,\beta)}\chi_{(r,1)}\|_{p,\lambda}\le C \|(1-t^2)^{-1/4}\chi_{(0,r)}\|_{p,\lambda}.
\]
By using \eqref{eq:MNT}, we obtain the inequality
\[
\int_{0}^{r} \frac{dt}{1-t} \|(1-t)^{-1/4}\chi_{(r,1)}\|_{p,\lambda}\le C\|(1-t)^{-1/4}\chi_{(0,r)}\|_{p,\lambda},
\]
which, by \eqref{eq:norm-Bola}, is equivalent to $-\log(1-r)\le C$ and this is impossible.

In the case $p=4(1-\lambda)/3$,  $W_{1,n}$ and $W_{3,n}$ are bounded. Let us suppose that $W_{2,n}$ is also bounded. Thus, we have the equivalent inequality
\[
\|(1-(\cdot))^{1/2}p_{n+1}^{(\alpha+1,\beta+1)}Hg\|_{p,\lambda}\le \|(p_{n}^{(\alpha,\beta)})^{-1}g\|_{p,\lambda}.
\]
Then we can consider \eqref{eq:Hil-2} with $s=1$, $r=1/2$ and $f(t)=(p_n^{(\alpha,\beta)}(t))^4(1-t)$, to obtain
\begin{multline}
\label{eq:No-pbajo}
\int_{-1}^{1/2} (p_n^{(\alpha,\beta)}(t))^4\, dt \|(1-(\cdot)^2)^{1/2}p_n^{(\alpha+1,\beta+1)}\chi_{(1/2,1)}\|_{p,\lambda}\\\le C \|(1-t)(p_{n}^{(\alpha,\beta)})^{3}\chi_{(-1,1/2)}\|_{p,\lambda}.
\end{multline}
It is clear that
\[
\|(1-t)(p_{n}^{(\alpha,\beta)})^{3}\chi_{(-1,1/2)}\|_{p,\lambda}\le C \|(1+t)^{1/4}\chi_{(-1,1/2)}\|_{p,\lambda}\le C,
\]
where in the last step we have used \eqref{eq:norm-Bola}. Following the estimate in \cite{Leves} to evaluate the $L^p$-norm of Jacobi polynomials, we have 
\[
\int_{-1}^{1/2} (p_n^{(\alpha,\beta)}(t))^4\, dt\sim \log n.
\]
Finally, we observe that, for any $\varepsilon >0$ and $n$ big enough,
\begin{align*}
\|(1-(\cdot)^2)^{1/2}p_n^{(\alpha+1,\beta+1)}\chi_{(1/2,1)}\|_{p,\lambda}
&\ge  \liminf_{n\to}\|(1-(\cdot)^2)^{1/2}p_n^{(\alpha+1,\beta+1)}\chi_{(1/2,1)}\|_{p,\lambda}-\varepsilon\\
&\ge \|(1-(\cdot)^2)^{1/4}\chi_{(1/2,1)}\|_{p,\lambda}-\varepsilon\ge C,
\end{align*}
by using again \eqref{eq:norm-Bola}. In this way, \eqref{eq:No-pbajo} implies $\log n\le C$ and this is not possible.
\end{proof}

\section{Proof of Lemmas}
\label{sec:aux}

\begin{proof}[Proof of Lemma \ref{lem:norms-p}]
The case $\lambda=0$ is well known and it can be done, with the proper modifications, following the ideas in \cite[Proposition 1]{Leves}, so we omit it.
From \eqref{eq:Jaco-cota-h}, it will be enough to study the norms $\|h_{\pm,n}\|_{p,\lambda}$. We will focus on $h_{-,n}$, the other can be analyzed in the same way.

We have
\[
\|h_{-,n}\|_{p,\lambda}=\sup_{x\in (-1,1),r>0}\left(r^{-\lambda}n^{p/2-2}I\right)^{1/p},
\]
with $I=\int_{s}^{t}(z+1)^{-p/4}\, dz$, $s=(1-x-r)n^2$ and $t=(1-x+r)n^2$. Note that in our situation we can consider $r$ bounded by a positive value. We will analyze different cases.

When $r\le n^{-2}$, with the estimate $I\le C(t-s)$ we have 
\[
r^{-\lambda}n^{p/2-2}I\le C r^{1-\lambda}n^{p/2}\le n^{p/2-2(1-\lambda)}.
\]

Now, when $r>n^{-2}$, we consider the cases $s\le 1 < t$ and $1\le s <t$ (when $s<t\le 1$ it is verified that $r\le n^{-2}$ and this case has been already treated).

If $s\le 1 < t$ and $r>n^{-2}$, for $p\le 4(1-\lambda)$, we have $I\le C t^{1-p/4}$. So, using that in this case $1-x\le C r$, we obtain 
\[
r^{-\lambda}n^{p/2-2}I\le C r^{-\lambda}(1-x+r)^{1-p/4}\le C r^{1-\lambda-p/4}\le C.
\]
For $p>4(1-\lambda)$, it is clear that $I\le C t^{\lambda}\int_{s}^{t}(1+z)^{-p/4-\lambda}\, dz\le C t^{-\lambda}$ and 
\[
r^{-\lambda}n^{p/2-2}I\le C r^{-\lambda}(1-x+r)^{\lambda}n^{p/2-2(1-\lambda)}\le C n^{p/2-2(1-\lambda)}.
\]

Finally, for $1\le s <t$ and $r<n^{-2}$, when $p\le 4(1-\lambda)$ we can check easily that $I\le C(t^{1-p/4}-s^{1-p/4})\simeq t^{-p/4}(t-s)$. Then 
\[
r^{-\lambda}n^{p/2-2}I\le C r^{1-\lambda}(1-x+r)^{-p/4}\le Cr^{1-\lambda-p/4}\le C.
\] 
For $p> 4(1-\lambda)$, we use that $I\le C t^{\lambda}\int_{s}^{t} z^{-p/4-\lambda} \simeq t^{\lambda}(s^{1-p/4-\lambda}-t^{1-p/4-\lambda})\simeq t^{\lambda-1}s^{1-p/4-\lambda}(t-s)$. In this way
\[
r^{-\lambda}n^{p/2-2}I\le C (1-x-r)^{1-p/4-\lambda}\le C n^{p/2-2(1-\lambda)},
\]
and the proof of the upper bound is completed.

To obtain the lower estimate in the case $p\le 4(1-\lambda)$ we use that, for $\varepsilon>0$ and $n$ big enough, by \eqref{eq:MNT},
\[
\|p_n^{(\alpha,\beta)}\|_{p,\lambda}\ge C \|(1-t^2)^{-1/4}\|_{p,\lambda}-\varepsilon.
\]
Then, the required estimate follows by \eqref{eq:norm-Bola}. In the case $p>4(1-\lambda)$, we will use that $|P_n^{(\alpha,\beta)}(x)|\ge n^{\alpha}$, for $1-\frac{1}{n^2}<x<1$. This is a well known consequence of the Hilb type formula for the Jacobi polynomials. So,
\[
\|p_n^{(\alpha,\beta)}\|_{p,\lambda}\ge C \Big(n^{2\lambda+p(\alpha+1/2)}\int_{1-1/n^2}^1(1-x)^{p\alpha/2}\, dx\Big)^{1/p}\ge C n^{1/2-2(1-\lambda)/p},
\]
and the proof is finished.
\end{proof}

\begin{proof}[Proof of Lemma \ref{lem:norms-q}]
We analyze $h_{-,n}$, the other case is similar. It is clear that
\begin{align*}
\|h_{-,n}\|_{q,\lambda}^{*}&\le
\inf_{x\in (1-1/n^2,1)}\int_0^\infty r^{\lambda/p-1} \|\chi_{(B(x,r))^c}h_{-,n}\|_{L^q(-1,1)}\, dr\\&=
\inf_{x\in (1-1/n^2,1)}(J_1+J_2),
\end{align*}
where
\[
J_1=\int_0^{1-x} r^{\lambda/p-1} \|\chi_{(-1,x-r)\cup (x+r,1)}h_{-,n}\|_{L^q(-1,1)}\, dr
\]
and
\[
J_2=
\int_{1-x}^{1+x} r^{\lambda/p-1} \|\chi_{(-1,x-r)}h_{-,n}\|_{L^q(-1,1)}\, dr.
\]
It is easy to check that $\|h_{-,n}\|_{L^q(-1,1)}$ can be controlled by the right hand side in \eqref{eq:norma-Lp} and then
\[
J_1\le C \|h_{-,n}\|_{L^q(-1,1)}\int_{0}^{1/n^2}r^{\lambda/p-1}\, dr\le C n^{-2\lambda/p}\|h_{-,n}\|_{L^q(-1,1)}\le C,
\]
where in the last step we have used the condition $p\ge 4(1-\lambda)/3$.

For $J_2$, when $q<4$ we have $J_2\le C$ because
\[
J_2\le C \|h_{-,n}\|_{L^q(-1,1)}\int_{1-x}^{1+x}r^{\lambda/p-1}\, dr\le C.
\]
In the case $q>4$, we have
\[
\|\chi_{(-1,x-r)}h_{-,n}\|_{L^q(-1,1)}\le C((1+n^{-2})^{1/q-1/4}+(1-x+r+n^{-2})^{1/q-1/4})
\]
and
\begin{align*}
J_2&\le C \int_{1-x}^{1+x} r^{\lambda/p-1} ((1+n^{-2})^{1/q-1/4}+(1-x+r+n^{-2})^{1/q-1/4})\, dr\\
&\le C\left(1+\int_{1-x}^{1+x} r^{\lambda/p-1} (1-x+r+n^{-2})^{1/q-1/4})\, dr\right)\\
&\le C\left(1+\int_{1-x}^{1+x} r^{\lambda/p+1/q-5/4}\, dr\right)\\
&\le C(1+(1+x)^{\lambda/p+1/q-1/4}+(1-x)^{\lambda/p+1/q-1/4})\le C,
\end{align*}
where in the last step we have used the condition $p\ge 4(1-\lambda)/3$.
When $q=4$,
\[
J_2\le C \int_{1-x}^{1+x} r^{3\lambda/4-1} \log\left(\frac{2n^2+1}{1+n^2(1-x+r)}\right)\, dr.
\]
Then, after applying integration by parts, we deduce that
\begin{align*}
J_2&\le C\left(1+n^2(1-x)^{3\lambda/4}+\int_{1-x}^{1+x} \frac{n^{2}r^{3\lambda/4}}{1+n^2(1-x+r)}\, dr\right)
\\& \le C \left(1+n^2(1-x)^{3\lambda/4}+\int_{1-x}^{1+x} r^{3\lambda/4-1}\, dr\right)\\& \le C (1+n^2(1-x)^{3\lambda/4}+(1-x)^{3\lambda/4})
\end{align*}
and $\inf_{x\in (1-1/n^2,1)}(J_1+J_2)\le C$.
\end{proof}

\begin{proof}[Proof of Lemma \ref{lem:Hil-Mario}]
If $g(t)=\chi_{(-1,r)}(t)|f(t)|$, for a function $f$ such that $vf\in\mathcal{L}_{p,\lambda}(-1,1)$, for $r\le r \le s$, we have
\[
|Hg(x)|=\int_{-1}^r \frac{|f(t)|}{x-t}\, dt \ge \int_{-1}^r \frac{|f(t)|}{s-t}\, dt
\]
and then
\[
|Hg(x)|\ge \chi_{(s,r)}(x)\int_{-1}^r \frac{|f(t)|}{s-t}\, dt.
\]
In this way, the boundedness of the Hilbert transform implies \eqref{eq:Hil-2}.
\end{proof}


\end{document}